\documentclass[12pt,a4paper,oneside]{amsart}
\pdfoutput=1
\usepackage{amsfonts, amsmath, amssymb, amsthm}
\usepackage[colorlinks=true,citecolor=blue]{hyperref}
\usepackage{txfonts}
\usepackage[T1]{fontenc}
\usepackage{bbm}
\usepackage{mathtools}
\usepackage{graphicx}
\usepackage{enumerate}
\usepackage{verbatim}
\usepackage{tikz}
\usepackage{algorithmic}
\usetikzlibrary{decorations,decorations.pathreplacing, math}

\newtheorem{theorem}{Theorem}[section]
\newtheorem*{theorem*}{Theorem}
\newtheorem{lemma}[theorem]{Lemma}
\newtheorem*{lemma*}{Lemma}

\newtheorem{corollary}[theorem]{Corollary}
\newtheorem{observation}[theorem]{Observation}

\theoremstyle{definition}

\newtheorem{remark}[theorem]{Remark}
\newtheorem*{remark*}{Remark}

\newcommand{\Symk}{S_{\hspace{-.3ex}k}}
\newcommand{\R}{\mathbb{R}}
\newcommand{\Sf}{\varphi}
\newcommand{\psf}{\eta}

\newcommand{\brak}[2]{\bigl[\!\begin{smallmatrix} #1 \\ #2 \end{smallmatrix}\!\bigr]}

\DeclareMathOperator{\conv}{conv}

\title{Colorful Intersections and Tverberg Partitions}

\author{Michael Gene Dobbins\hspace{0.03cm}$^\ast$}
\thanks{$^\ast$ Department of Mathematics and Statistics, Binghamton University, Binghamton, New York, USA. Support from KAIST Advanced Institute for Science-X (KAI-X)}

\author{Andreas F. Holmsen$\hspace{0.03cm}^\dagger$}

\author{Dohyeon Lee$\hspace{0.03cm}^\dagger$}
\thanks{$^\dagger$ Department of Mathematical Sciences, KAIST, Daejeon, South Korea  \and 
Discrete Mathematics Group, Institute for Basic Science (IBS), Daejeon, South Korea. Supported by the Institute for Basic Science (IBS-R029-C1)}
\thanks{Email: mdobbins@binghamton.edu,  andreash@kaist.edu,  mathloveguy@kaist.ac.kr}


\begin{document}

\begin{abstract} The colorful Helly theorem and Tverberg's theorem are fundamental results in discrete geometry. We prove a theorem which interpolates between the two. In particular, we show the following for any integers $d \geq m \geq 1$ and $k$ a prime power. Suppose $F_1, F_2, \dots, F_m$ are families of convex sets in $\mathbb{R}^d$, each of size $n > (\frac{d}{m}+1)(k-1)$, such that for any choice $C_i\in F_i$ we have $\bigcap_{i=1}^mC_i\neq \emptyset$. Then, one of the families $F_i$ admits a Tverberg $k$-partition. That is, one of the $F_i$ can be partitioned into $k$ nonempty parts such that the convex hulls of the parts have nonempty intersection. As a corollary, we also obtain a result concerning $r$-dimensional transversals to families of convex sets in $\R^d$ that satisfy the colorful Helly hypothesis, which extends the work of Karasev and Montejano.
\end{abstract}

\maketitle

\section{Introduction}

\subsection{Background} A {\em $k$-partition} of a finite set $X$ is an {\em ordered} partition of $X$ into $k$-nonempty parts, that is, a $k$-tuple
\[(X_1, X_2, \dots, X_k), \]
such that $X = X_1\cup X_2 \cup \cdots \cup X_k$ and $X_i\neq \emptyset$ for all $i$. If the elements of $X$ are points (or convex sets) in $\mathbb{R}^d$, a $k$-partition of $X$, $(X_1, X_2, \dots, X_k)$, is called a {\em Tverberg $k$-partition} provided that 
\[(\conv\: X_1) \cap (\conv\:X_2) \cap \cdots \cap (\conv\:X_k) \neq \emptyset.\]

A fundamental result of discrete geometry is the celebrated theorem of Tverberg \cite{tverberg-66}, which  asserts that if $|X|>(d+1)(k-1)$, then $X$ admits a Tverberg $k$-partition.

\medskip

Let $F_1, F_2, \dots, F_{m}$ be families of sets. We say that 
$F_1, F_2, \dots, F_{m}$ satisfy {\em the colorful intersection property} if 
\[C_1\cap C_2 \cap \cdots \cap C_{m} \neq \emptyset\] for every choice $C_1\in F_1, C_2\in F_2, \dots, C_{m}\in F_{m}$. 

Another fundamental result of discrete geometry is the colorful Helly theorem \cite{barany-82}.
It asserts that if $F_1, F_2, \dots, F_{d+1}$ are finite families of convex sets in $\R^d$ that satisfy the colorful intersection property, then there is a point in common to every member of one of the families. Observe that it is no loss in generality to assume that $|F_i|=n$ for all $i$, in which case the conclusion asserts (in our terminology) that one of the $F_i$ admits a Tverberg $n$-partition. 

\medskip

Tverberg's theorem and the colorful Helly theorem have both played important roles in the development of discrete geometry, and there are a number of generalizations and extensions. For further information on the subject, we suggest the reader consult \cite{barany-21, BK-21, lgmm2019, Eckhoff1993, holmsen2017helly, mato-dg-2002} and the references therein.

\medskip

In the last decade, a particular intriguing question related to the colorful Helly theorem has been under investigation: {\em What conclusions (if any) can be drawn if we are given fewer than $d+1$ families of convex sets in $\R^d$ which satisfy the colorful intersection property?} See \cite{km2011, mrr2020, mont2013} for some answers to this question, as well as \cite[Conjecture 1]{mrr2020} and \cite[Problems 8.1 and 8.2]{BK-21}.

\subsection{Main results} Our main result can be viewed as an interpolation between the colorful Helly theorem and Tverberg's theorem.

\begin{theorem}\label{thm:tverberg}
Given integers $d\geq m \geq 1$ and $k$ a prime power. Suppose 
$F_1, F_2, \dots, F_m$ are families of convex sets in $\R^d$, with $|F_i| = n > (\frac{d}{m}+1)(k-1)$, that satisfy the colorful intersection property. Then one of the $F_i$ admits a Tverberg $k$-partition.
\end{theorem}

\begin{remark}\label{rem:projection}
Note that for $m=d$, our theorem follows immediately from the colorful Helly theorem. First project the sets into a hyperplane and apply the colorful Helly theorem. There is a point in common to all the members of one of the projected families, and the preimage of this point is a line that intersects all the members of one of the $F_i$. Now, just apply Tverberg's theorem within this line. 

For smaller $m$ this argument becomes generally less effective. Projecting into an $(m-1)$-flat and applying the colorful Helly theorem would give a $(d-m+1)$-flat transversal to one of the families $F_i$, but to obtain a Tverberg $k$-partition within this $(d-m+1)$-flat would require $|F_i| > (d-m+2)(k-1)$, which is significantly worse than the size of $F_i$ that we require in Theorem \ref{thm:tverberg}. 
\end{remark}

\begin{remark} \label{rem:construction}
Whenever $d $ is a multiple of $m$, our theorem is optimal with respect to the size of the families. Indeed, consider the following construction in $\R^d = \R^t \times \R^t \times \cdots \times \R^t$  with $m$ copies of $\R^t$, where $t=\frac{d}{m}$. Let $X\subset \mathbb{R}^t$ be a set of $(t+1)(k-1)$ points which does not admit a Tverberg $k$-partition, and for $1\leq i \leq m$ define the family
\[F_i := \big\{ \R^t \times \cdots \times \R^t \times \underset{i\text{th factor}}{\{x\}} \times \R^t \times \cdots \times \R^t : x \in X \big\}.\] 
Observe that these families satisfy the colorful intersection property, but none of them has a Tverberg $k$-partition.
\end{remark}

\begin{remark}\label{rem:conjectures}
The prime power assumption in our Theorem \ref{thm:tverberg} appears to be just an artifact of our proof method, and we conjecture that the theorem also holds when $k$ is any positive integer.

Another interesting instance that our proof method falls short of would be to generalize Theorem \ref{thm:tverberg} to the case where the $F_i$ may have distinct sizes and we ask for one of the $F_i$ to admit a Tverberg $k_i$-partition. 
We leave it to the reader's imagination to formulate (and prove) such a conjecture.
\end{remark}

\subsection{An application to geometric transversals} 
Some of the earliest results dealing with the colorful intersection property for less than $d+1$ families of convex sets in $\R^d$ are the transversal theorems of Karasev and Montejano \cite{km2011, mont2013}. The simplest case asserts: {\em Given three red convex and three blue convex sets in $\R^3$, where each red set intersects each blue set, then there is a line that intersects every red set or a line that intersects every blue set.} More generally, we have the following

\begin{corollary} \label{cor:transversal}
Let $F_1, F_2, \dots, F_m$ be families of convex sets in $\mathbb{R}^d$, each of size $k+r$,  satisfying the colorful intersection property, where $k$ is a prime power. If $d< \frac{(r+1)m}{k-1}$, then one of the $F_i$ have an $r$-dimensional affine flat transversal. 
\end{corollary}

\begin{proof} Theorem \ref{thm:tverberg} implies that one of the $F_i$ admits a Tverberg $k$-partition. That is, for $F_i = \{C_1, C_2, \dots, C_{k+r}\}$, there is a $k$-partition $(X_1, X_2, \dots, X_k)$ of $[k+r]$ and a point $x\in \R^d$ such that  
$x\in \conv \big(\bigcup_{i\in X_j} C_i\big)$ for every $1\leq j \leq k$. This means that we can choose a point $x_i\in C_i$, for every $1\leq i \leq k+r$, such that $x\in \conv\{x_i\}_{i\in X_j}$ for every $1\leq j \leq k$. Let $A_j$ denote the affine hull of $\{x_i\}_{i\in X_j}$, and let $A$ be the affine hull of $A_1 \cup A_2 \cup \cdots \cup A_k$. Then $x\in A_1\cap A_2\cap \cdots \cap A_k$, and so
\[\dim A \leq \sum_{j=1}^k \dim A_j = \sum_{j=1}^k (|X_j|-1) = r.\] Since $x_i\in A$ for every $1\leq i \leq k+r$, it follows that $A$ is an affine flat of dimension at most $r$ that intersects every $C_i$.
\end{proof}

\begin{remark} \label{rem:transversal}
We note that the work of Karasev and Montejano deals with the cases $k=2$ \cite[Theorem 8]{km2011} and $k=m$ \cite[Corollary 7]{km2011}, but without any primality condition on $m$. For $k=2$, their theorem requires $d < r+m+1$, while Corollary \ref{cor:transversal} allows for $d < (r+1)m$. However, for the case $k=m$, their theorem requires $d < (\frac{r}{m-1}+1)m$ which is slightly better than our bound $d < (\frac{r+1}{m-1})m$ given by Corollary \ref{cor:transversal}.
\end{remark}

\begin{remark}
The work of Karasev and Montejano makes use of Schubert calculus and the Lusternik–Schnirelmann category of the Grassmannian.
In contrast, 
our proof uses a combination of the configuration space\:/\:test map scheme from topological combinatorics (see e.g. \cite{mato-bu-2003}) 
and {Sarkaria's tensor method} from discrete geometry (see e.g. \cite[section 8.3]{mato-dg-2002}). 
The latter appear more frequently in discrete and computational geometry literature, so we expect these methods to attract broader interest from the community. 
\end{remark}

\section{Proof of Theorem \ref{thm:tverberg}}

We will show that a hypothetical counter-example to Theorem~\ref{thm:tverberg} would contradict the following

\begin{theorem*}[Volovikov \cite{volov-96}]
Let $G = \mathbb{Z}_p \times \mathbb{Z}_p \times \cdots \times  \mathbb{Z}_p $ be the product of finitely many copies, with $p$ prime. Let $X$ and $Y$ be fixed point free $G$-spaces, where $X$ is $n$-connected and $Y$ is finite dimensional and homotopy equivalent to $S^n$. Then there is no $G$-equivariant map $X \to Y$. 
\end{theorem*}

To reach a contradiction, we use 
{Sarkaria's tensor method} 
to get geometric criteria for the existence of a Tverberg $k$-partition, 
which we then use to construct such an equivariant map.   
We then use discrete Morse theory to show that the domain of this map is sufficiently connected to violate Volovikov's theorem. 

\subsection{The configuration space \texorpdfstring{$K_{n,k}^{*m}$}{}}

Given integers $n > k \geq 1$, we let $V_{n,k}$ denote the set of {\em surjective} maps \[\varphi: [n] \to [k].\] Note that we can equivalently think of  $V_{n,k}$ as the set of $k$-partitions on $[n]$ 
\[\big( \varphi^{-1}(1), \dots, \varphi^{-1}(k)\big).\] 
For a family of sets $F = \{C_1, \dots, C_n\}$ and $\varphi \in V_{n,k}$, we write $\varphi F$ to denote the $k$-partition
\[\big( F^{(1)}, \dots,  F^{(k)} \big)\] where $F^{(i)} = \{C_j \; :  \; j\in \varphi^{-1}(i)\}$.

We define
$K_{n,k}$ to be the simplicial complex on the vertex set $V_{n,k}$ whose faces consists of subsets $\sigma = \{\varphi_1, \dots, \varphi_r\} \subset V_{n,k}$ such that 
\begin{equation}\label{eq:Xinonempty}
X_i = \varphi_1^{-1}(i) \cap \cdots \cap \varphi_r^{-1}(i) \neq \emptyset
\end{equation}
for every $i\in [k]$. 
Note that the facets of $K_{n,k}$ correspond to a choice of representative from each part of a $k$-partition, and the vertices of a facet correspond to all the ways of extending this choice of representatives to a $k$-partition. 

The symmetric group $\Symk$ acts (freely) on $V_{n,k}$ by permuting the parts of the partition.  That is, for $g\in \Symk$ and $\varphi = \big( \varphi^{-1}(1), \dots, \varphi^{-1}(k)\big) \in V_{n,k}$ we have
\[g\varphi = \big( \varphi^{-1}(g(1)), \dots, \varphi^{-1}(g(k))\big),\]
which means that $\Symk$ acts freely on $K_{n,k}$.

In order to apply Volovikov's theorem, we need a lower bound on the connectedness of $K_{n,k}$. This is given by

\begin{lemma} \label{lem:pcomplex}
For all $n > k \geq 1$, the simplicial complex $K_{n,k}$ is 
$(n-k-1)$-connected. 
\end{lemma}

We give the proof of this lemma in section \ref{sec:pcomplex}.
For now we proceed with the proof of Theorem \ref{thm:tverberg} assuming the bound on the connectedness of $K_{n,k}$.
We will construct an equivariant map on the $m$-fold join $K_{n,k}^{*m} = K_{n,k}*\dots*K_{n,k}$. 

\subsection{Sarkaria's criterion} Here we demonstrate one of the standard methods for proving Tverberg's theorem, Sarkaria's tensor method. The method is usually applied for collections of points, but here we apply it to families of convex sets, similar to the approach  taken in \cite{sarkar22}.

For each $i\in [k]$, define the vector $v_i\in \mathbb{R}^k$ as 
\[v_i = e_i - \textstyle{\frac{1}{k}}\mathbf{1}, \] where $e_i$ is the $i$th standard unit vector and $\mathbf{1} = (1,\dots, 1) \in \mathbb{R}^k$.  Observe that the $v_i$ satisfy only one linear dependency up to a scalar multiple, which is 
\begin{equation}\label{eq:lindep}
v_1 + \cdots + v_k = 0.
\end{equation}
Next, define the map
\[
\begin{array}{cccc}
L_i :  & \mathbb{R}^d & \to & \mathbb{R}^{(d+1)\times k}\\
& x & \mapsto & \brak{x}{1} \otimes v_i
\end{array}
\]

where $\brak{x}{1}$ denotes the vector in $\mathbb{R}^{d+1}$ obtained from $x$ by appending an additional coordinate and setting this equal to 1. For a given $x\in \mathbb{R}^d$, we will regard the image $L_i(x)$ as a $(d+1)\times k$ matrix. Observe that since each $v_i$ is orthogonal to $\mathbf{1} \in \mathbb{R}^k$, it follows that $L_i(x)$ belongs to the subspace
\[Y=\{\:[w_1 \; \cdots \; w_k]\: : \: w_1+\dots+w_k=\mathbf{0}\} \subset \mathbb{R}^{(d+1)\times k}.\] We note that the symmetric group $\Symk$ acts on the subspace $Y$ by permuting the columns, and for $g\in \Symk$ and $L_i(x)\in Y$, we have 
\[gL_i(x) = g(\textstyle \brak{x}{1} \otimes v_i) = \textstyle \brak{x}{1} \otimes v_{g(i)} = L_{g(i)}(x).\] Observe that the action is not free, but it is {\em fixed-point free} on $Y\setminus\{0\}$. 

For a convex set $C\subset \mathbb{R}^d$, we write $L_iC$ to denote the set
\[L_iC = \{L_i(x) \; : \; x\in C\},\] which is a convex subset of $Y$. The crucial step of  the Sarkaria method is the following

\begin{observation}\label{obs:sarkaria}
Let $F = \{C_1, \dots, C_n\}$ be a family of convex sets in $\mathbb{R}^d$ and $\varphi\in V_{n,k}$. If  
\[0 \in \text{\em conv} \left( L_{\varphi(1)}C_1  \cup \cdots \cup L_{\varphi(n)}C_n \right),\] 
then 
$\varphi F$ is a Tverberg $k$-partition.
\end{observation}

Indeed, suppose $0 = \alpha_1L_1(x_1) + \cdots  + \alpha_kL_k(x_k)$ is a convex combination, where 
\[L_i(x_i) = \textstyle{ \brak{x_i}{1}} \otimes v_i \in \conv  \left( \bigcup_{j\in \varphi^{-1}(i)} L_iC_j \right).\]
Then in each coordinate, the $\alpha_i\brak{x_i}{1}$ are the coefficients of a linear dependency of the $v_i$, and since \eqref{eq:lindep} is the unique linear dependency up to scalar multiple, we have 
\[\alpha_1 = \cdots = \alpha_k \;\;\; \text{ and } \;\;\; x_1 = \cdots = x_k.\]
We also have $x_i \in \conv  \left( \bigcup_{j\in \varphi^{-1}(i)} C_j \right)$, 
so $\varphi F$ is a Tverberg $k$-partition.

\subsection{The test map \texorpdfstring{$f$}{}} 
Consider a single family $F_1 = \{C_1, \dots, C_n\}$ of compact convex sets in $\mathbb{R}^d$ which {\em does not} have a Tverberg $k$-partition. We show how this gives us an equivariant map
\[f_1 :K_{n,k} \to Y.\]
For $\varphi \in V_{n,k}$, consider the $k$-partition $\varphi F_1$. By hypothesis, this is not a Tverberg $k$-partition, so by Observation \ref{obs:sarkaria}, there is a vector $a_{\varphi} \in Y$ which defines an open halfspace 
\[H_{\varphi} = \{y \in Y : a_{\varphi} \cdot y > 0 \}\] such that 
\begin{equation}\label{eq:halfspaces}
 L_{\varphi(1)}C_{1} \cup   \cdots \cup L_{\varphi(n)}C_{n} \subset H_{\varphi}.
\end{equation}
It is important that the vectors $a_{\varphi}$ are chosen such that 
\[ga_{\varphi} = a_{g\varphi}\]
for every $g\in \Symk$. This can be done by first choosing one vector $a_{\varphi}$ in each $\Symk$ orbit, and then allowing the rest of the vectors in that orbit to be defined accordingly.

To verify that such a choice is valid, suppose that $a_{\varphi}$ has been chosen such that the containment \eqref{eq:halfspaces} holds and consider $g\in \Symk$.  Since $gL_i(x) = L_{g(i)}(x)$, we get
\[
\begin{array}{rclcl}
L_{g\varphi(i)}C_i & = & g(L_{\varphi(i)}C_i)  &\subset & g(H_{\varphi}) \\
& & & = & \{ g(y) :  a_\varphi \cdot y >0 \} \\
& & & = & \{z : a_\varphi \cdot g^{-1}(z)>0\} \\
& & & = & \{z : ga_\varphi \cdot z >0\} \\
& & & = & \{z : a_{g\varphi}\cdot z > 0\} = H_{g\varphi}.
\end{array}
\]
The equivariant map $f_1: K_{n,k} \to Y$ is defined by affine extension of $a_{\bullet}$ by setting 
\[f_1(\sigma) = \conv \{a_\varphi : \varphi \in \sigma\},\]
for every face $\sigma\in K_{n,k}$.

\medskip

Now consider the setting of Theorem \ref{thm:tverberg} where we have families $F_1, \dots, F_m$ of compact convex sets in $\mathbb{R}^d$, each of size $n$, and suppose none of them have a Tverberg $k$-partition. For every $1\leq i \leq m$, the family $F_i$ gives us an equivariant map $f_i : K_{n,k} \to Y$ as defined above. 
By taking joins, we get an equivariant map 
\[f  = f_1 * f_2 * \cdots * f_m : K_{n,k}^{*m} \to Y.\] 
(Note that the symmetric group $\Symk$ acts freely on $K^{*m}_{n,k}$ by applying the group action to each component of the join.)

\subsection{Using colorful intersections to avoid a subspace} 
We now want to determine the range of $f$, and for this we use the assumption that the families $F_i$ satisfy the colorful intersection property. 

First we note that the facets of $K_{n,k}$ are in one-to-one correspondence with {\em injective} functions 
\[\rho: [k] \to [n].\]
In particular, each facet $\sigma \in K_{n,k}$ corresponds to the unique injection $\rho : [k] \to [n]$ that has $\sigma$ as its set of left inverses.  That is,
\[\varphi \in \sigma \iff \varphi \circ \rho = \text{id}_{[k]}.\]

Consider a facet $\sigma  = \sigma_1 * \sigma_2 *\cdots * \sigma_m \in K^{*m}_{n,k}$ with the corresponding injections $\rho_i$ satisfying $\varphi \circ \rho_i = \text{id}_{[k]}$ for all $\varphi \in \sigma_i$. For each $1\leq i \leq m$, we apply $\rho_i$ to select a $k$-tuple of distinct convex sets 
\[ C^{(i)}_{\rho_i(1)} , C^{(i)}_{\rho_i(2)} , \dots, C^{(i)}_{\rho_i(k)}  \in F_i. \] 
The colorful intersection property now guarantees that, for every $1\leq j \leq k$, we can select a point $x_j\in \mathbb{R}^d$ which satisfies
\[x_j \in  C^{(1)}_{\rho_1(j)}  \cap  C^{(2)}_{\rho_2(j)}  \cap  \cdots  \cap  C^{(m)}_{\rho_m(j)}.\] 

Now consider a vertex $\varphi \in \sigma_i$ and the halfspace $H_\varphi = \{y\in Y : a_\varphi \cdot y >0\}$ from the definition of the test map $f$.
Since $\varphi \circ \rho_i = \text{id}_{[k]}$, we get 
\[L_1C^{(i)}_{\rho_i(1)}  \cup L_2C^{(i)}_{\rho_i(2)} \cup \cdots \cup L_kC^{(i)}_{\rho_i(k)} \subset H_\varphi,    \]
by the containment \eqref{eq:halfspaces}, which in turn implies
\[\big\{ L_1(x_1), L_2(x_2), \dots, L_k(x_k) \big\} \subset H_\varphi. \]
Thus, for every vertex $\varphi \in \sigma$ and every $1\leq j \leq k$, we have
$a_\varphi \cdot L_j(x_j) > 0$, which gives us
\begin{equation}\label{eq:facet}
f(\sigma) = \conv \{a_\varphi : \varphi\in \sigma \} \subset \{y\in Y : L_j(x_j) \cdot y > 0\}.
\end{equation}

We claim that $f(\sigma)$ does not intersect the subspace 
\[B := e_{d+1}\otimes\R^k.\]  
For the sake of contradiction, suppose there were a vector $c = (c_1, c_2, \dots, c_k) \in \R^k$ such that 
\[ b = e_{d+1} \otimes c \in f(\sigma) \cap B.\]
Let $j\in [k]$ be a coordinate such that $c_j$ is minimized.  Then 
\[ v_j \cdot c = \textstyle  \frac{k-1}{k} c_j - \frac{1}{k} \sum_{i\neq j} c_i \leq 0,\]
and so
\[\textstyle L_j(x_j) \cdot b = \left(  \brak{x_j}{1} \otimes v_j \right)  \cdot  \left( e_{d+1} \otimes c \right) =\left( \brak{x_j}{1} \cdot e_{d+1} \right) \left(v_j \cdot c \right) \leq 0,\]
but $L_j(x_j) \cdot b > 0$ by \eqref{eq:facet} since $b \in f(\sigma)$, so $f(\sigma)$ cannot intersect $B$.

We conclude that the range of $f$ is contained in $Y\setminus B$, so we have an equivariant map
\[f : K^{*m}_{n,k} \to Y\setminus B.\]

\subsection{Finishing the proof} 
By Lemma \ref{lem:pcomplex}, it follows that the complex $K^{*m}_{n,k}$ is $(m(n-k+1)-2)$-connected (see e.g. \cite[Proposition 4.4.3]{mato-bu-2003}), and since $n > (\frac{d}{m}+1)(k-1)$, we get 
\[m(n-k+1)-2 >  d(k-1)-2.\]
Therefore $K^{*m}_{n,k}$ is at least $(d(k-1)-1)$-connected. 

The symmetric group $\Symk$ acts on $Y\setminus B$ by permuting columns, which makes it a fixed-point free action. If $k = p^r$, then the action is also fixed-point free with respect to the subgroup $G = \mathbb{Z}_p \times \mathbb{Z}_p \times \cdots \times \mathbb{Z}_p$. The subspace $Y\cap B^\perp$ has dimension $d(k-1)$ as it consists of the matrices in $Y$ whose $(d+1)$st row is equal to the 0-vector, and therefore $Y\setminus B$ is homotopy equivalent to $S^{d(k-1)-1}$. Consequently,  the existence of the map $f$ contradicts Volovikov's theorem. \qed

\section{Proof of Lemma \ref{lem:pcomplex}} \label{sec:pcomplex}

Here we show  that the simplicial complex $K_{n,k}$ is $(n-k-1)$-connected. The proof is in two steps. First we  define a polyhedral complex $C_{n,k}$ whose cells correspond to {\em partial surjective functions} $\pi : [n] \to [k]$, and use Quillen's fiber lemma to show that $C_{n,k}$ is homotopy equivalent to $K_{n,k}$. We then bound the connectedness of $C_{n,k}$ using discrete Morse theory.

\subsection{The complex of partial surjections}
Given integers $n > k \geq 1$, we let $C_{n,k}$ denote the set of {\em partial surjective functions} 
\[\psf : [n] \to [k].\]
Equivalently, we can think of $C_{n,k}$ as the set of $k$-partitions
\[\big( \psf^{-1}(1), \psf^{-1}(2), \dots, \psf^{-1}(k)\big)\]   
of {\em subsets} of $[n]$. Furthermore, we may identify an element $\psf \in C_{n,k}$ with the product of simplices
\[\Delta_\psf := 2^{\psf^{-1}(1)} \times 2^{\psf^{-1}(2)} \times \cdots \times 2^{\psf^{-1}(k)}, \] whose geometric realization is a convex polytope of dimension $|\psf^{-1}([k])| - k$. 

Observe that for $\eta, \gamma \in C_{n,k}$, if $\eta^{-1}([k]) \subset \gamma^{-1}([k])$  and $\eta(x) = \gamma(x)$ for all $x\in \eta^{-1}([k])$, then $\Delta_\eta$ is a face of $\Delta_{\gamma}$. Consequently, $C_{n,k}$ has the structure of a polyhedral complex
(see Figure \ref{fig:complexes}). Note that  $C_{n,k}$ may also be viewed as a poset where the faces are ordered by inclusion.
\begin{figure} \centering
 \begin{tikzpicture}

\begin{scope}[yshift = 7cm, xshift = -3cm]
\draw (0,0) --++ (3,-1) --++ (2.5,1) --++ (-1.5,2) --++ (-3,1) -- cycle ;
\draw (3,-1) --++ (1,3)  ;
\draw[opacity = 0.3] (0,0) --++ (2.5,1) --++ (-1.5,2)  (2.5,1) --++ (3,-1) ;

\filldraw[black] (0,0) circle (1pt);
\filldraw[black] (3,-1) circle (1pt);
\filldraw[black] (5.5,0) circle (1pt);
\filldraw[black] (4,2) circle (1pt);
\filldraw[black] (1,3) circle (1pt);
\filldraw[black, opacity =0.4] (2.5,1) circle (1pt);

\node[left] at (0,0) {\tiny $ 2 \: | \: 3 \: | \: 4 $}; 
\node[above] at (1,3) {\tiny $ 1 \: | \: 3 \: | \: 4 $}; 
\node[right, opacity =0.4] at (2.35,1.25) {\tiny $ 5 \: | \: 3 \: | \: 4 $};
\node[below] at (3,-1) {\tiny $ 2 \: | \: 3 \: | \: 6 $}; 
\node[right] at (5.5,0) {\tiny $ 5 \: | \: 3 \: | \: 6 $}; 
\node[above] at (4.5,1.9) {\tiny $ 1 \: | \: 3 \: | \: 6 $}; 

\fill[opacity=0.2, blue] (0,0) --++ (3,-1) --++ (2.5,1) --++ (-1.5,2) --++ (-3,1) ;
\fill[opacity=0.2, blue] (3,-1) --++ (2.5,1) --++ (-1.5,2);
\end{scope}

\begin{scope}[yshift = 0cm]
\filldraw[black] (90:1) circle (1pt);
\filldraw[black] (210:1) circle (1pt);
\filldraw[black] (330:1) circle (1pt);

\filldraw[black] (60:3) circle (1pt);
\filldraw[black] (120:3) circle (1pt);
\filldraw[black] (180:3) circle (1pt);
\filldraw[black] (240:3) circle (1pt);
\filldraw[black] (300:3) circle (1pt);
\filldraw[black] (360:3) circle (1pt);

\filldraw[black] (30:6) circle (1pt);
\filldraw[black] (150:6) circle (1pt);
\filldraw[black] (270:6) circle (1pt);

\draw[black] (270:6) arc[start angle=-60, end angle=0, radius= {sqrt(108)} ];
\draw[black] (30:6) arc[start angle=60, end angle=120, radius= {sqrt(108)} ];
\draw[black] (150:6) arc[start angle=180, end angle=240, radius= {sqrt(108)} ];

\fill[blue, opacity=0.1] (270:6) arc[start angle=-60, end angle=0, radius= {sqrt(108)} ];
\fill[blue, opacity=0.1] (30:6) arc[start angle=60, end angle=120, radius= {sqrt(108)} ];
\fill[blue, opacity=0.1](150:6) arc[start angle=180, end angle=240, radius= {sqrt(108)} ];
\fill[blue, opacity=0.1] (30:6) -- (150:6) -- (270:6);

\draw (90:1) -- (210:1) -- (330:1) --cycle;
\draw (0:3)-- (60:3)-- (120:3)-- (180:3)-- (240:3)-- (300:3)--cycle;

\draw (330:1)-- (360:3)-- (30:6) (90:1)-- (120:3)-- (150:6) (210:1)-- (240:3)--(270:6) ;
\draw (90:1)-- (60:3)-- (30:6) (210:1)-- (180:3)-- (150:6) (330:1)-- (300:3)--(270:6) ;

\node at (0:0) {\tiny $ 4 \: | \: 123$};

\node at (30:1.6) {\tiny $ 24 \: | \: 13$};
\node at (90:2) {\tiny $ 234 \: | \: 1$};
\node at (150:1.6) {\tiny $ 34 \: | \: 12$};
\node at (210:2) {\tiny $ 134 \: | \: 2$};
\node at (270:1.6) {\tiny $ 14 \: | \: 23$};
\node at (330:2) {\tiny $ 124 \: | \: 3$};

\node at (30:3.75) {\tiny $ 2 \: | \: 134$};
\node at (90:3.5) {\tiny $ 23 \: | \: 14$};
\node at (150:3.75) {\tiny $ 3 \: | \: 124$};
\node at (210:3.5) {\tiny $ 13 \: | \: 24$};
\node at (270:3.75) {\tiny $ 1 \: | \: 234$};
\node at (330:3.5) {\tiny $ 12 \: | \: 34$};

\node at (-10:5.5) {\tiny $ 123 \: | \: 4$};

\node at (90:1.25) {\fontsize{3.5}{4.5}\selectfont $ 4 \, | \, 1$};
\node at (61.5:3.15) {\fontsize{3.5}{4.5}\selectfont $ 2 \, | \, 1$};
\node at (118.5:3.15) {\fontsize{3.5}{4.5}\selectfont $ 3 \, | \, 1$};
\node at (30:6.25) {\fontsize{3.5}{4.5}\selectfont $ 2 \, | \, 4$};

\node at (210:1.25) {\fontsize{3.5}{4.5}\selectfont $ 4 \, | \, 2$};
\node at (181.5:3.2) {\fontsize{3.5}{4.5}\selectfont $ 3 \, | \, 2$};
\node at (238.5:3.2) {\fontsize{3.5}{4.5}\selectfont $ 1 \, | \, 2$};
\node at (150:6.25) {\fontsize{3.5}{4.5}\selectfont $ 3 \, | \, 4$};

\node at (330:1.25) {\fontsize{3.5}{4.5}\selectfont $ 4 \, | \, 3$};
\node at (301.5:3.2) {\fontsize{3.5}{4.5}\selectfont $ 1 \, | \, 3$};
\node at (358.5:3.2) {\fontsize{3.5}{4.5}\selectfont $ 2 \, | \, 3$};
\node at (270:6.2) {\fontsize{3.5}{4.5}\selectfont $ 1 \, | \, 4$};

\end{scope}

\end{tikzpicture}
 
\caption{Above: The polyhedral cell of $C_{6,3}$ corresponding to the ordered partition $\big( \{1,2,5\}, \{3\}, \{4,6\} \big)$, which we express more succinctly by $(125 \: | \: 3 \: | \: 46)$.  Below: The cell complex $C_{4,2}$. }
\label{fig:complexes}
\end{figure}
We need the following.

\begin{lemma}\label{lem:homotopeq}
For all integers $n> k \geq 1$, we have a homotopy equivalence $K_{n,k} \simeq C_{n,k}$.
\end{lemma}

\begin{proof}
Let $K=K_{n,k}$ and $C=C_{n,k}$. 
Our goal is to construct an order-reversing map 
\[g: K \to C \] such that $g^{-1}(C_{\succeq \eta})$ is contractible for every $\eta \in C$, 
where $C_{\succeq\eta} = \{\gamma\in C:\Delta_\gamma \supset \Delta_\eta\}$. 
This will establish the desired homotopy equivalence by Quillen's fiber lemma \cite[Theorem 10.5]{bjorner}. 

To this end, consider a face $\sigma = \{  \Sf_1, \Sf_2, \dots, \Sf_{|\sigma|} \}\in K_{n,k}$ and define the $k$-tuple
\[X = (X_1, X_2, \dots, X_k), \] 
where $X_j = \Sf^{-1}_1(j) \cap \Sf^{-1}_2(j) \cap \cdots \cap \Sf^{-1}_{|\sigma|}(j)$ as in \eqref{eq:Xinonempty}. 
By definition, each $X_j$ is nonempty, and so $X$ is a $k$-partition of a subset of $[n]$.  We may therefore identify $X$ with a partial surjective function $\psf_\sigma : [n] \to [k]$
given by $\psf_\sigma(i) = j$ when 
$\Sf(i)=j$ for each $\Sf \in \sigma$, 
and $\psf_\sigma(i)$ is undefined otherwise,  
or equivalently 
$\psf_\sigma^{-1}(j) = X_j$. 
By setting 
\[g(\sigma) := \eta_\sigma,\]
we have  
\begin{align*} \label{eq:monotone}
\tau \subset \sigma 
&\implies 
\psf_\sigma^{-1}(j) \subset \psf_\tau^{-1}(j) \text{ for each $j\in[k]$} \\
&\implies 
\Delta_{g(\sigma)} 
\subset \Delta_{g(\tau)},
\end{align*}
so we obtain a surjective map $g : K \to C$, which is order-reversing. 
The upper set $C_{\succeq \psf}$ is the set of partial functions $\gamma$ 
such that $\gamma^{-1}(j) \supset \psf^{-1}(j)$ for each $j$, 
that is, $C_{\succeq \psf}$ is the set of extensions of $\psf$, so 
\[g^{-1}(C_{\succeq \psf}) = 2^{\sigma_\psf}\] where 
\[\sigma_\psf : = \big\{ \Sf \in K_{n,k} : \Sf(i) = \psf(i) \text{ if } \eta(i) \text{ is defined} \big\}. \]
Thus the fibers are simplices, 
and the associated complexes $K_{n,k}$ and $C_{n,k}$ 
are homotopy equivalent by Quillen's fiber lemma.
\end{proof}

\subsection{Acyclic matchings}

We now apply discete Morse theory \cite{forman} to determine the connectedness of $C_{n,k}$.
This means that we now view $C_{n,k}$ as a poset, and it contains  the {\em empty face} as a unique minimal element. 

\medskip

Recall that a {\em matching} in a poset $P$ is a matching in the underlying graph of the Hasse diagram of $P$. In other words, a matching $M$ in $P$ is a collection of pairs 
\[ M = \big\{\{a_1,b_1 \}, \{a_2, b_2\}, \dots, \{a_t,b_t\}\big\}, \] where the 
$a_i$ and $b_j$ are all distinct, and $a_i$ is an immediate predecessor of $b_i$ for every $i$.
The matching $M$ is called {\em cyclic} if there is a subsequence of indices $i_1, i_2, \dots, i_s \in [t]$ such that 
\[\begin{array}{lcl}
    a_{i_2} & \prec & b_{i_1} \\
    a_{i_3} & \prec & b_{i_2} \\
     & \vdots & \\
    a_{i_s} & \prec & b_{i_{s-1}} \\
    a_{i_1} & \prec & b_{i_s}
\end{array}
\]
If no such subsequence exists, then the matching $M$ is called {\em acyclic}.

\medskip

In the case when $P$ is the face lattice of a polyhedral complex and $M$ is an acyclic matching, then the unmatched elements of $P$ are called the {\em critical cells} of the matching.
One of the fundamental theorems of discrete Morse theory \cite[Theorem 4.7]{jonsson} (see also \cite[Theorem 11.13]{kozlov}) asserts that if all the critical cells of $M$ have dimension at least $d$, then the polyhedral complex is $(d-1)$-connected.
Our goal is therefore to show the following
\begin{lemma} \label{lem:acyclic}
For all $n> k \geq 1$, there is an acyclic matching $M$ on $C_{n,k}$ whose critical cells (if there are any) have dimension $n-k$. 
\end{lemma}

In order to find such an acyclic matching we need two basic tools from the toolbox of discrete Morse theory. The first one (referred to as an {\em element matching} \cite[Lemma 4.1]{jonsson}) describes a particular acyclic matching which we use repeatedly. 
Let $(\sigma-x,\sigma+x)$ denote $(\sigma\setminus\{x\},\sigma\cup\{x\})$. 
\begin{lemma*}[Element matching] \label{lem:element}
Let $X$ be a finite set, $P\subseteq 2^X$ is ordered by inclusion, and for a fixed element $x\in X$, let
\begin{align}
P_x 
&= \{\sigma : \sigma-x, \sigma+x \in P\} \\
M_x 
&= \{ \{\sigma-x, \sigma+x\} : \sigma \in P_x \}.
\end{align}
Then $M_x$ is an 
acyclic matching on $P$. 
\end{lemma*}

The second tool allows us to combine a collection of acyclic matchings into a single one (see \cite[Lemma 4.2]{jonsson} or  \cite[Theorem 11.10]{kozlov}). 
\begin{lemma*}[Patchwork lemma] \label{lem:cluster}
Let $P$ and $Q$ be finite posets, and let $h: P \to Q$ be an order-preserving map. Assume we have acyclic matchings $M_q$ on each of the subposets $h^{-1}(q)$. Then, $M =  \bigcup_{q\in Q} M_q$ is an acyclic matching on $P$.
\end{lemma*}

\begin{proof}[Proof of Lemma  \ref{lem:acyclic}]
We proceed by induction on $k$. For $k=1$, observe that $C_{n,1}$ is isomorphic to $2^{[n]}$, and so the element matching 
\[ \big\{ \{\sigma-n, \sigma+n\} : \sigma \in 2^{[n]} \big\} \]
is a complete acyclic matching (i.e. there are no critical cells).

Now, assume $k>1$ and that the lemma holds for $C_{n',k-1}$ for all $n' > k-1$. We denote the cells of $C_{n,k}$ as $k$-tuples
\[X = (X_1, X_2, \dots, X_k),\]
where the $X_i$ are either nonempty, pairwise disjoint subsets of $[n]$ (corresponding to a nonempty cell of $C_{n,k}$), or they satisfy $X_1 = X_2 = \cdots = X_k = \emptyset$ (corresponding to the unique empty cell).
Note that the dimension of a nonempty cell equals $\sum_{i=1}^k(|X_i|-1)$, and 
so our goal is to find an acyclic matching in $C_{n,k}$ whose critical cells (if there are any) satisfy $X_1 \cup X_2 \cup \cdots \cup X_k = [n]$. 

\medskip

Let  
$\{a_1, a_2, \dots, a_{k-1}\}$ be an antichain and let
$a_k \prec a_i$ for all $1\leq i<k$.
Then a map $h_1 : C_{n,k} \to \{a_1, a_2, \dots, a_k\}$, defined by
\[h_1(X_1, X_2, \dots, X_k) = 
\begin{cases}
a_i & \text{ if } n\in X_i \text{ and } i < k \\
a_k & \text{ otherwise } 
\end{cases}
\]
is order-preserving. Our goal is therefore to apply the patchwork lemma by finding appropriate acyclic matchings on each of the subposets $A_i := h_1^{-1}(a_i)$. 

\medskip

Fix $1\leq i < k$ and consider the projection map $p_i: A_i \to (2^{[n-1]})^{\times (k-1)}$ which forgets the $i$th component
\[p_i(X_1, X_2, \dots, X_k) = (X_1, \dots, X_{i-1}, X_{i+1}, \dots, X_k).\]
This is an order preserving map when the range of $p_i$ is ordered by (componentwise) inclusion. For a given $(X_1, \dots, X_{i-1}, X_{i+1}, \dots, X_k)$ in the range of $p_i$, we observe that $p_i^{-1}(X_1, \dots, X_{i-1}, X_{i+1}, \dots, X_k)$ is isomorphic to $2^{Y_i}$, where \[Y_i := [n-1]\setminus (X_1\cup \cdots \cup X_{i-1}\cup X_{i+1}\cup \cdots \cup X_k).\] Therefore, if $Y_i\neq \emptyset$, then $p_i^{-1}(X_1, \dots, X_{i-1}, X_{i+1}, \dots, X_k)$ has a complete element matching. Otherwise, we have
\[p_i^{-1}(X_1, \dots, X_{i-1}, X_{i+1}, \dots, X_k) = (X_1, \dots, X_{i-1}, \{n\}, X_{i+1}, \dots, X_k),\] which will be one of our critical cells of dimension $n-k$. 

\medskip

It remains to find an appropriate acyclic matching on $A_k$. Define an order-preserving map 
$h_2 : A_k \to \{a \prec b\}$ where 
\[h_2(X_1, X_2, \dots, X_k) = 
\begin{cases}
a & \text{ if } X_k = \{n\} \text{ or } X_k = \emptyset,  \\
b & \text{ otherwise. }
\end{cases}\]
Set $A := h_2^{-1}(a)$ and $B: = h_2^{-1}(b)$. Note that the empty cell belongs to $A$, and $A$ is isomorphic to $C_{n-1, k-1}$. By induction, there is an acyclic matching on $A$ such that all critical cells (if there are any) have dimension $n-k$. 

Consider the projection map $p_k: B \to (2^{[n-1]})^{\times (k-1)}$ which forgets the $k$th component
\[p_i(X_1, X_2, \dots, X_k) = (X_1, X_2, \dots,  X_{k-1}).\]
This is an order preserving map when the range of $p_k$ is ordered by (componentwise) inclusion. For a given $(X_1, X_2, \dots, X_{k-1})$ in the range of $p_k$, we observe that $p_i^{-1}(X_1, X_2, \dots, X_{k-1})$ is isomorphic to $2^{Y_k}\setminus \big\{\emptyset, \{n\} \big\}$, where \[Y_k := [n]\setminus (X_1\cup X_2 \cup \cdots \cup X_{k-1}).\] By definition, $Y_k \neq \emptyset$, and so $p_i^{-1}(X_1, X_2, \dots, X_{k-1})$ has a complete element matching of the form $\{\sigma -n, \sigma +n\}$. 
\end{proof}


\bibliographystyle{plain}

\end{document}